\documentclass[reqno, oneside, 12pt]{amsart}

\usepackage[letterpaper]{geometry}
\usepackage{amsmath}
\usepackage{amssymb}
\usepackage{enumerate}
\usepackage{verbatim}
\usepackage{mathrsfs}
\usepackage{mathtools}

\mathtoolsset{centercolon}

\usepackage{setspace}

\numberwithin{equation}{section}

\newtheorem{thm}{Theorem}[section]

\theoremstyle{remark}
\newtheorem*{rem*}{Remark}

\date{}

\author{Heidi Goodson}
\address{Department of Mathematics and Statistics, Haverford College; 370 Lancaster Avenue, Haverford, PA 19041 USA}
\email{hgoodson@haverford.edu}

\title[Hypergeometric Properties of Genus 3 Generalized Legendre Curves]{Hypergeometric Properties of Genus 3 Generalized Legendre Curves}

\begin{document}

\begin{abstract}
Inspired by a result of Manin, we study the relationship between certain period integrals and the trace of Frobenius of genus 3 generalized Legendre curves. We show that both of these properties can be computed in terms of ``matching" classical and finite field hypergeometric functions, a phenomenon that has also been observed in elliptic curves and many higher dimensional varieties.
\end{abstract}

\maketitle

\section{Introduction}

The motivation for this work comes from a particular family of elliptic curves. For $\lambda\not= 0,1$ we define an elliptic curve in the Legendre family by 
$$E_\lambda : y^2 = x(x-1)(x-\lambda).$$

We compute a period integral associated to the Legendre elliptic curve given by integrating the nowhere vanishing holomorphic $1$-form $\omega=\frac{dx}{y}$ over a $1$-dimensional cycle containing $\lambda$. This period is a solution to a hypergeometric differential equation and can be expressed as the classical hypergeometric series
\begin{equation}\label{eqn:ECperiod}
\pi=\int_{0}^{\lambda} \frac{dx}{y}={}_2F_1\left(\left.\begin{array}{cc}
                \frac12&\frac12	\\
		{}&1
               \end{array}\right|\lambda\right).\end{equation}
See the exposition in \cite{Clemens} for more details on this. \\

Koike \cite[Section 4]{Koike1995} showed that, for all odd primes $p$ and $\lambda\in\mathbb Q\setminus\{0,1\}$, the trace of Frobenius for curves in this family can be expressed in terms of Greene's finite field hypergeometric function
\begin{equation}\label{eqn:EC2F1Greene}
a_{E_\lambda}(p)=-\phi(-1)p\cdot {}_2F_1\left(\left.\begin{array}{cc}
                \phi&\phi\\
		{}&\epsilon
               \end{array}\right|\lambda\right)_{p},\end{equation}
where $\epsilon$ is the trivial character and $\phi$ is a quadratic character modulo $p$.\\ 

Note the similarity between the period and trace of Frobenius expressions: the period is given by a classical hypergeometric series whose arguments are the fractions with denominator 2 and the trace of Frobenius is given by a finite field hypergeometric function whose arguments are characters of order 2. This similarity is to be expected for curves by the following result of Manin. 
\begin{thm}\cite[Theorem 2]{Manin}\label{thm:Manin}
The rows of the Hasse-Witt matrix satisfy every differential relation which is satisfied by the elements of a basis of the space of differentials of the first kind (dual to a canonical basis of $H^1(M,\mathscr O)$) modulo the space of complete differentials.
\end{thm}
For elliptic curves, the Hasse-Witt matrix has a single entry: the trace of Frobenius. In Corollary 3.2 of \cite{GoodsonDwork2017}, we show that the matching finite field and classical ${}_2F_1$ hypergeometric expressions in Equations \ref{eqn:ECperiod} and \ref{eqn:EC2F1Greene} are congruent modulo $p$ for odd primes. This result would imply merely a congruence between the finite field hypergeometric function expression and the trace of Frobenius. The fact that Koike showed that we actually have an equality is very intriguing and leads us to wonder for what other varieties this type of equality holds.\\

Further examples of a correspondence between arithmetic properties of varieties and finite field hypergeometric functions have been observed for algebraic curves  
and for Calabi-Yau manifolds. 
For example, Fuselier \cite{Fuselier10} gave a finite field hypergeometric trace of Frobenius formula for elliptic curves with $j$-invariant $\frac{1728}{t}$, where $t\in \mathbb F_p \setminus \{0,1\}$.  Lennon \cite{Lennon1} extended this by giving a hypergeometric trace of Frobenius formula that does not depend on the Weierstrass model chosen for the elliptic curve. In \cite{AhlgrenOno00a}, Ahlgren and Ono gave a formula for the number of $\mathbb F_p$ points on a modular Calabi-Yau threefold. 
We extended this work in \cite{GoodsonDwork2017, GoodsonDwork2017b} by showing that the number of points on Dwork hypersurfaces over finite fields can be expressed in terms of Greene's finite field hypergeometric functions.\\

In this paper we examine the connection between analytic and arithmetic properties of algebraic curves. We approach this story from two directions. First, in Sections \ref{sec:Lefschetz} and \ref{sec:HasseWitt}, we develop tools that are needed to understand Manin's statement in Theorem \ref{thm:Manin}. Then, in Section \ref{sec:GenLegendreCurve}, we apply Manin's theory to the family of genus 3 generalized Legendre curves
$$C^4_{\lambda}:\; y^4=x(x-1)(x-\lambda).$$
In this example, we see Manin's theory in action since we obtain ``matching" hypergeometric expressions for the analytic and arithmetic data associated to these curves. We begin with background information in Section \ref{sec:HGF}.

\section{Preliminaries}
\label{sec:HGF}
In this section we recall definitions and properties of ${}_2F_1$ classical and finite field hypergeometric functions. See, for example, \cite{GoodsonDwork2017, Greene, Slater1966} for extensions of this work to ${}_{n+1}F_n$ generalized hypergeometric functions.\\

We define the classical hypergeometric series by
\begin{equation}\label{eqn:classicalHGF}
 {}_{2}F_{1}\left(\left.\begin{array}{cc}
                a&b\\
		{}&c
               \end{array}\right|x\right) = \displaystyle\sum_{k=0}^{\infty}\dfrac{(a)_k(b)_k}{(c)_kk!}x^k,
\end{equation}
where $(\alpha)_0=1$ and $(\alpha)_k=\alpha(\alpha+1)(\alpha+2)\ldots(\alpha+k-1)$.  \\

This series is a solution to the so-called hypergeometric differential equation
\begin{equation}\label{eqn:hypdiffeq}
    -abF+(c-(a+b+1)x)\frac{d}{dx}F+x(1-x)\frac{d^2}{dx^2}F=0
\end{equation}
(see, for example, \cite[Section 1.2]{Slater1966}).\\

Unless either $a$ or $b$ is a negative integer, classical hypergeometric series have an infinite number of terms. In some cases, for example when considering congruences or supercongruences, we may only need to consider a finite number of these terms. For a positive integer, $m$, we define the hypergeometric series truncated at $m$ to be 
\begin{equation}\label{eqn:truncHGF}
 {}_{2}F_{1}\left(\left.\begin{array}{cc}
                a&b\\
		{}&c
               \end{array}\right|x\right)_{\text{tr}(m)} = \displaystyle\sum_{k=0}^{m-1}\dfrac{(a)_k(b)_k}{k!(c)_k}x^k.
\end{equation}

In his 1987 paper \cite{Greene}, Greene introduced a finite field, character sum analogue of classical hypergeometric series. Let $\mathbb F_q$ be the finite field with $q$ elements, where $q$ is a power of an odd prime $p$. If $\chi$ is a multiplicative character of $\widehat{\mathbb F_q^{\times}}$, extend it to all of $\mathbb F_q$ by setting $\chi(0)=0$. For any two characters $A,B$ of $\widehat{\mathbb F_q^{\times}}$ we define the normalized Jacobi sum by

\begin{equation}\label{eqn:normalizedjacobi}
 \binom{A}{B}:=\frac{B(-1)}{q}\sum_{x\in\mathbb F_q} A(x)\overline B(1-x) = \frac{B(-1)}{q}J(A,\overline{B}),
\end{equation}
where $J(A,B)=\sum_{x\in \mathbb F_q} A(x)B(1-x)$ is the usual Jacobi sum.\\

For any positive integer $n$ and characters $A,\; B,\; C$ in $\widehat{\mathbb F_q^{\times}}$, Greene defined the finite field hypergeometric function ${}_{2}F_1$ over $\mathbb F_q$ by
\begin{equation}\label{eqn:HGFdef}
 {}_{2}F_{1}\left(\left.\begin{array}{cc}
                A&B\\
		{} &C
               \end{array}\right|x\right)_q = \displaystyle\frac{q}{q-1}\sum_{\chi}\binom{A\chi}{\chi}\binom{B\chi}{C\chi}\chi(x).
\end{equation}
An alternate definition, which is in fact Greene's original definition, is given by
\begin{equation}\label{eqn:2F1def}
 {}_{2}F_{1}\left(\left.\begin{array}{cc}
                A&B\\
		{} &C
               \end{array}\right|x\right)_q = \epsilon(x)\frac{BC(-1)}{q}\sum_yB(y)\overline{B}C(1-y)\overline{A}(1-xy).
\end{equation}

Greene's finite field hypergeometric functions were defined independently of those by Katz \cite{Katz1990} and McCarthy \cite{McCarthy2012c}, though relations between them have been demonstrated in \cite{McCarthy2012c}.\\

Greene shows that defining finite field hypergeometric functions in this way leads to many transformation properties that mirror those of classical series. For example, classical ${}_2F_1$ hypergeometric series satisfy the following identity \cite[p. 48]{Slater1966}
\begin{align*}
    {}_{2}F_1\left(\left.\begin{array}{cccc}
                -m&b\\
		{}&c
               \end{array}\right|x\right)&=\frac{(b)_m}{(c)_m}(-x)^m {}_{2}F_1\left(\left.\begin{array}{cccc}
                -m&1-c-m\\
		{}&1-b-m
               \end{array}\right|\frac1x\right)
\end{align*}
The analogous statement for finite field hypergeometric functions is as follows

\begin{thm}\cite[Theorem 4.2, ii]{Greene}\label{thm:Greenetransform}
 For characters $A, B, C$ of $\mathbb F_q$ and $x\in\mathbb F_q^\times$,
 \begin{align*}
     {}_{2}F_{1}\left(\left.\begin{array}{cc}
                A&B\\
		{} &C
               \end{array}\right|x\right)_q &=ABC(-1)\overline A(x){}_{2}F_{1}\left(\left.\begin{array}{cc}
                A&A\overline C\\
		{} &A\overline B
               \end{array}\right|\frac{1}{x}\right)_q 
 \end{align*}
\end{thm}

Note that these identities can be generalized to ${}_{n+1}F_n$ classical and finite field hypergeometric functions for $n>1$. See Section 4 of \cite{Greene} for other transformation and summation theorems.\\

In addition to having analogous transformation properties, ``matching" classical and finite field hypergeometric functions have also been shown to be congruent modulo $p$ in many cases. The following theorem will be referenced in our discussion in Section \ref{sec:periodpointcount}.
\begin{thm}\cite[Theorem 3.1]{GoodsonDwork2017}\label{theorem:2F1congruence}
 Let $m$ and $d$ be integers with $1\leq m<d$. If $p\equiv 1\pmod d$ and $T$ is a generator for the character group $\widehat{\mathbb F_p^{\times}}$ then, for $x\not=0$,
 \begin{equation*}
   {}_{2}F_{1}\left(\left.\begin{array}{cc}
                \tfrac{m}{d}&\tfrac{d-m}{d}\\
		{} &1
               \end{array}\right|x\right)_{\text{tr}(p)} \equiv  -p\hspace{.05in}{}_{2}F_{1}\left(\left.\begin{array}{cc}
                T^{mt}&\overline T^{mt}\\
		{} &\epsilon
               \end{array}\right|x\right)_p \pmod p,
 \end{equation*}
where $t=\tfrac{p-1}{d}$.
\end{thm}
This builds on supercongruence results of Mortenson \cite{Mortenson2003a, Mortenson2005} by considering hypergeometric functions evaluated away from 1, though this result holds mod $p$ instead of $p^2$. Further congruences and supercongruences between classical and finite field hypergeometric functions can be found in \cite{Swisher2016,GoodsonDwork2017}.

\section{Using the Lefschetz Number}\label{sec:Lefschetz}
The Lefschetz number associated to a map from a manifold to itself essentially keeps track of the number of fixed points of the map. Let $f:M\rightarrow M$ be a differentiable map on the compact differentiable manifold $M$ such that the graph of $f$ meets the diagonal transversely. 
Then the Lefschetz number $L(f)$ can be computed in two ways:
\begin{equation}\label{Lefschetz}
 L(f)=\sum_{p\in M} \sigma_p(f) = \sum_{n=0}^\infty (-1)^n\text{tr}[f^*:H^n(M,\mathbb C) \rightarrow H^n(M,\mathbb C)],
\end{equation}
where 
\[ \sigma_p(f)=\left\{
  \begin{array}{rl}
    0 & : f(p)\not=p\\
    \pm 1 & : (\text{graph } f) \text{ meets diagonal with positive/negative orientation}.
  \end{array}
\right.\]

When the map $f$ is the Frobenius map on a curve, then $L(f)$ measures the number of points on the curve over a finite field $\mathbb F_q$. In this field we have $(x,y)=(x^q, y^q)$, so that any point on the curve will be a fixed point of the map.\\

We will rewrite both expressions for the Lefschetz number in order to show the relationship between the period associated to a curve and its point count.  We follow the work of Clemens \cite[Chapter 2]{Clemens}.\\ 

We start by rewriting $\sum_{p\in M} \sigma_p(f)$. Let $J_p(f)$ be the Jacobian of $f$ at the point $p$. The transversality of $f$ at $p$ implies that $(\text{identity}- f)$ has maximal rank at $p$. This is the rank of $I-J_p(f)$ at the point $p$, which is a matrix that gives us information about the orientation of the map $f$. Thus, we can write $\sigma_p(f)=\text{sign }\det{(I-J_p(f))}$. Clemens shows that this determinant can also be expressed as 
\begin{equation}\label{eqn:DetJp}
    \det{(I-J_p(f))} = \sum_{r=0}^n(-1)^r\text{tr}(\wedge^rJ_p(f))
\end{equation}
so that we can write 
$$\sum_{p\in M} \sigma_p(f)= \sum_{p,r}(-1)^r\frac{\text{tr}(\wedge^rJ_p(f))}{|\det{(I-J_p(f))} |}.$$

Denote the restrictions of $J_p(f)$ to type (1,0) (holomorphic) and type (0,1) (anti-holomorphic) parts of $J_p(f)$ by $J_p'(f)$ and $J_p''(f)$, respectively. Clemens notes that if the manifold $M$ is a K{\"a}hler manifold then we can replace the de Rham complex by the Dolbeault complex on $M$. Thus Equation \ref{Lefschetz} becomes
\begin{equation}\label{Lefschetz1}
 \sum_{p,r}(-1)^r\frac{\text{tr}(\wedge^rJ''_p(f))}{|\det{(I-J_p(f))} |} = \sum_{n=0}^\infty (-1)^n\text{tr}[f^*|_{H^n(M,\mathscr O)}],
\end{equation}

We also have that  
$$\sum_{r}(-1)^r\text{tr}(\wedge^rJ''_p(f)) = \det{(I-J_p''(f))},$$
as we did in Equation \ref{eqn:DetJp}, and $$\det{(I-J_p(f))}=\det{(I-J_p'(f))}\det{(I-J_p''(f))}.$$
Hence,
\begin{eqnarray*}
 \sum_{p,r}(-1)^r\frac{\text{tr}(\wedge^rJ''_p(f))}{|\det{(I-J_p(f))} |}&=&\sum_{p}\frac{\det{(I-J_p''(f))}}{|\det{(I-J_p'(f))}\det{(I-J_p''(f))}|}\\
								      &=&\sum_{p}\frac{1}{|\det{(I-J_p'(f))}|}\\
								      &=&\sum_{p \text{ fixed}}\frac{1}{|\det{(I-J_p'(f))}|}.
\end{eqnarray*}
Thus, $\sum\sigma_p(f)$ can be expressed in terms of the holomorphic part of $J_p(f)$.\\

We specialize to the case where $f$ is the Frobenius map and the manifold is an algebraic curve $C$. Note that $J_p(f)=0$ since $d(x^p)/dx = px^{p-1}=0$ in $\mathbb F_p$. Hence, $|\det{(I-J_p'(f))}|=1$ and $\sum_{p\in C} \sigma_p(f)=$ the number of fixed points of $f$. Since $f$ is a map on $C$ and $x^p=x$ if and only if $x\in\mathbb F_p$, the number of fixed points of $f$ will be exactly the number of $\mathbb F_p$-points on $C$ plus the point at infinity. Thus, $$\sum_{p\in C} \sigma_p(f)= 1+\text{ the number of $\mathbb F_p$-points on }C.$$

We now rewrite the expression $\sum_{n=0}^\infty (-1)^n\text{tr}[f^*|_{H^n(C,\mathscr O)}]$ for the case we are considering. Recall that, in general,  $H^n(M,\mathscr O)=0$ whenever $n>\dim(M)$. Equation \ref{Lefschetz1} then becomes
\begin{equation*}
 1+\text{ the number of $\mathbb F_p$-points on $C$}= 1-\text{tr}[f^*|_{H^1(C,\mathscr O)}],
\end{equation*}
i.e.
\begin{equation}\label{eqn:PointCount}
 \text{ the number of $\mathbb F_p$-points on $C$}= -\text{tr}[f^*|_{H^1(C,\mathscr O)}].
\end{equation}

In Section \ref{sec:HasseWitt} we will see that the right-hand-side of this equation is related to the periods of an algebraic curve.

\section{The Hasse-Witt Matrix}\label{sec:HasseWitt}
In this section, we piece together the work of \cite{Clemens} and \cite{Manin}. Let $g$ be the genus of the algebraic curve $C$. The Hasse-Witt matrix of $C$ is the $g\times g$ matrix of the Frobenius map with respect to a basis of regular differentials of the first kind. Thus, the trace of this matrix will give us $\text{tr}[f^*|_{H^1(C,\mathscr O)}]$ (trace is independent of basis). In this section we aim to describe the Hasse-Witt matrix in greater detail.\\

The genus $g$ of the curve is equal to both the dimension of the space $H^1(C,\mathscr O)$ of 1-cycles and the dimension of the space of regular 1-forms on $C$. We will choose dual bases for these two spaces (dual with respect to a residue pairing).  Let $P_1, \ldots, P_g$ be a set of distinct points on $C$ such that the divisor $D=\sum P_i$ is nonspecial. It is noted by Manin \cite[Section 1.5]{Manin} that we may identify $H^1(C,\mathscr O)$ with the space of functions that have poles at worst at the points $P_1, \ldots, P_g$. Thus, we can choose a basis $h_1, \ldots, h_g$ for $H^1(C,\mathscr O)$, where each $h_i$ is a function with a simple pole at $P_i$ and no other poles (except at infinity). Thus, $h_i$ has Taylor series expansion,
$$h_i=\frac{1}{x-x_i}+\sum_{l\geq0}c_{i,l}(x-x_i)^l,$$
where $P_i=(x_i,y_i)$ is a point on $C$ as above \cite[Section 2.12]{Clemens}. Similarly, to each point $P_i=(x_i,y_i)$ we can associate a differential $\omega_i$, which is to say we can write $\omega_i$ locally at the point $P_i$:
$$\omega_i=dx+\sum_{r\geq1}a_{i,r}(x-x_i)^r dx.$$
The bases $\{\omega_i\}_i$ and $\{h_i\}_i$ are dual with respect to the pairing ${(\omega_i,h_j)=\text{Res}(h_j\omega_i,P_i)}$, the residue at $P_i$, since 
\[\text{Res}(h_j\omega_i,P_i)=\left\{
  \begin{array}{rl}
    1 & \text{if } i=j\\
    0 &  \text{if }i\not=j.
  \end{array}
\right.\]
Let $K$ be the matrix of scalar products $\left[(\omega_i, h_j)\right]$. We can write the Hasse-Witt matrix $H$ as
$$H=KH=\left[(\omega_i, f^*h_j)\right],$$
where the map $f^*$ sends each $h_i(x)$ to
$$h_i(x^p)=\frac{1}{(x-x_i)^p}+\sum_{l\geq0}b_{i,l}(x-x_i)^{pl}.$$
Thus, $\text{tr}[f^*|_{H^1(C,\mathscr O)}]=\sum_{i=1}^g (\omega_i, f^*h_i)$. In fact we can say even more about this matrix. Note that if $i\not=j$ then
$$(\omega_i, f^*h_j)=\text{Res}(f^*h_j\omega_i,P_i)=0$$
since $h_j$, and therefore $f^*h_j$, is holomorphic at the point $P_i$. Thus, the Hasse-Witt matrix is a diagonal matrix with this choice of basis.\\

These diagonal entries can be expressed in terms of coefficients in the expansions of the differentials. We have that $\text{Res}(f^*h_j\omega_i,P_i)$ is the coefficient of ${1}/{(x-x_i)}$ in the expansion
$$f^*h_i\omega_i=\left(\frac{1}{(x-x_i)^p}+\sum_{l\geq0}b_{i,l}(x-x_i)^{pl}\right)\left(1+\sum_{r\geq1}a_{i,r}(x-x_i)^r\right)dx.$$
Thus, $(\omega_i, f^*h_i)=a_{i,p-1}$, so that $\text{tr}[f^*|_{H^1(C,\mathscr O)}]=\sum_{i=1}^g a_{i,p-1}$.\\

\section{Generalized Legendre Curves}\label{sec:GenLegendreCurve}

We now apply this theory to a particular family of curves. We look at a specific case of  generalized Legendre curves given by
$$C^4_{\lambda}: y^4=x(x-1)(x-\lambda).$$
When viewed as a projective curve, it is given by the homogeneous equation
$$Y^4=ZX(X-Z)(X-\lambda Z)$$
by sending $(x,y) \to (X/Z,Y/Z)$. When written in this form we see that the curve is nonsingular in $\mathbb P^2$. Thus, by a well-known genus formula for nonsingular curves, the genus of $C^4_{\lambda}$ is $g=\frac{(4-1)(4-2)}{2}=3$.

\subsection{Period Computation}\label{subsec:Period}
In this section we give formulas for certain period integrals associated to genus 3 generalized Legendre curves.  The periods we are interested in are obtained by choosing dual bases of the space of holomorphic differentials and the space of cycles ${H^1(C_\lambda^4,\mathscr O)}$ and integrating the differentials over each cycle. Note that Barman and Kalita developed a hypergeometric formula for one of these period integrals in \cite{Barman2012} using trigonometric substitution.\\

Using the method described in \cite[Section 2]{Archinard2002} yields the following basis for the space of differentials
$$\left\{\omega_1=\frac{x dx}{y^3},\; \omega_2=\frac{dx}{y^2},\; \omega_3=\frac{dx}{y^3}\right\}.$$

\begin{thm}\label{thm:GenLegendreCurvePeriod}
The periods of the genus 3 generalized Legendre curve are
$$\pi_1={}_2F_1\left(\left.\begin{array}{cc}
                1/4&3/4\\
		{}&1/2
               \end{array}\right|\lambda\right), \; \pi_2={}_2F_1\left(\left.\begin{array}{cc}
                1/2&1/2\\
		{}&1
               \end{array}\right|\lambda\right),\; \pi_3={}_2F_1\left(\left.\begin{array}{cc}
                3/4&5/4\\
		{}&3/2
               \end{array}\right|\lambda\right).$$
\end{thm}

\begin{proof}

As noted in Section \ref{sec:HasseWitt},  we can write each $\omega_i$ locally at a distinct point $P_i$ on the curve.  We compute the periods $\pi_1, \pi_2,\pi_3$ of $C^4_{\lambda}$ by integrating each differential $\omega_i$ over a cycle in ${H^1(C^4_{\lambda},\mathscr O)}$ that contains the point $P_i$ and not the other $P_j$. Such a cycle exists since the chosen points are distinct.\\

We follow the work of Clemens \cite[Section 2.10]{Clemens} to find differential equations satisfied by the periods and then give combinatorial expressions for them. We show the computation for $\pi_3$ and omit the work for the remaining periods. Starting with the differential $$\omega_3=\frac{dx}{y^3}=(x(x-1)(x-\lambda))^{-3/4}dx,$$ we take derivatives with respect to $\lambda$ to get
\begin{eqnarray*}
  \frac{\partial}{\partial \lambda}((x(x-1)(x-\lambda))^{-3/4})&= -\frac34x^{-3/4}(x-1)^{-3/4}(x-\lambda)^{-7/4}\\\\
  \frac{\partial^2}{\partial \lambda^2}((x(x-1)(x-\lambda))^{-3/4})&= \frac{21}{16}x^{-3/4}(x-1)^{-3/4}(x-\lambda)^{-11/4}.
  \end{eqnarray*}

We wish to find a linear combination of $\omega_3$ and its derivatives that gives an exact differential. To do this, we rewrite the following differential
\begin{align*}
   d\left(\frac{x^{1/4}(x-1)^{1/4}(x-\lambda)^{1/4}}{(x-\lambda)^2}\right)&=d\left(x^{1/4}(x-1)^{1/4}(x-\lambda)^{-7/4}\right)\\
   &=\left[\frac14x^{-3/4}(x-1)^{1/4}(x-\lambda)^{-7/4}\right.\\
   &\hspace{.3in}\left.+\frac14x^{1/4}(x-1)^{-3/4}(x-\lambda)^{-7/4}-\frac74x^{1/4}(x-1)^{1/4}(x-\lambda)^{-11/4}\right]\\
   &=\frac13(x-1)\frac{d\omega_3}{d\lambda}+\frac13x\frac{d\omega_3}{d\lambda}-\frac43x(x-1)\frac{d^2\omega_3}{d\lambda^2}\\
   &= -\frac54\omega_3-2(2\lambda+1)\frac{d\omega_3}{d\lambda}-\frac43\lambda(\lambda-1)\frac{d^2\omega_3}{d\lambda^2}.
  \end{align*}
By integrating both sides and then multiplying by 3/4, we see that $\pi_3$ satisfies $F_3\pi_3=0$, where
\begin{equation}\label{de3}
 F_3=-\frac{15}{16}+(3/2-3\lambda)\frac{d}{d\lambda}+\lambda(1-\lambda)\frac{d^2}{d\lambda^2}.
\end{equation}

Note that this is a hypergeometric differential equation. We solve for $a,b,c$ in Equation \ref{eqn:hypdiffeq} and find that $a=3/4, b=5/4$ (or vice versa) and $c=3/2$. This gives us the following expression for the period
$$\pi_3={}_2F_1\left(\left.\begin{array}{cc}
                3/4&5/4\\
		{}&3/2
               \end{array}\right|\lambda\right).$$

Similarly, we find that $\pi_1$ satisfies $F_1\pi_1=0$, where
\begin{equation}
 F_1=-\frac{3}{16}+(1/2-2\lambda)\frac{d}{d\lambda}+\lambda(1-\lambda)\frac{d^2}{d\lambda^2}
\end{equation}
and can be expressed as
$$\pi_1={}_2F_1\left(\left.\begin{array}{cc}
                1/4&3/4\\
		{}&1/2
               \end{array}\right|\lambda\right),$$
and that $\pi_2$ satisfies $F_2\pi_2=0$, where
\begin{equation}
 F_2=-\frac{1}{4}+(1-2\lambda)\frac{d}{d\lambda}+\lambda(1-\lambda)\frac{d^2}{d\lambda^2}
\end{equation}
and can be expressed as
$$\pi_2={}_2F_1\left(\left.\begin{array}{cc}
                1/2&1/2\\
		{}&1
               \end{array}\right|\lambda\right).$$

\end{proof}
              
\subsection{Point Count}\label{subsec:PointCount}
In this section we will compute the number of points on the curve $C^4_\lambda$ in two ways. We first specify the work in Section \ref{sec:HasseWitt} to the curve $C^4_\lambda$. Then, we will compute the number of points using character sums.\\

Recall Manin's result (see Theorem \ref{thm:Manin} of this paper) that tells us that the rows of the Hasse-Witt matrix satisfy every differential equation satisfied by the periods of a curve. In Section \ref{sec:HasseWitt} we specifically chose bases for the spaces of differentials and cycles that were dual to each other, which results in a diagonal Hasse-Witt matrix. Thus, the sum of the rows of the Hasse-Witt matrix is exactly the trace of the matrix, which in this case is the trace of Frobenius. Hence, the trace of Frobenius must satisfy the same differential equations as the periods.\\

Moreover, since the space of differentials is 3-dimensional and we have developed differential equations for three $\mathbb C$-linearly independent periods, it must be the case that the trace of Frobenius is a $\mathbb C$-linear combination of the periods:
$$\text{tr}[f^*|_{H^1(M,\mathscr O)}]\equiv \sum _{i=1}^3 c_{i,p}(\lambda) \pi_{i} \pmod{p},$$
where each $ c_{i,p}(\lambda)\in\mathbb C$ may depend on the order $p$ of the field and on the parameter $\lambda$ of the curve. Using Equation \ref{eqn:PointCount} we can conclude that
$$\text{number of $\mathbb F_p$-points on } C^4_\lambda\equiv -\sum _{i=1}^3 c_{i,p}(\lambda) \pi_{i} \pmod{p},$$
which we showed in Section \ref{subsec:Period} is $\mathbb C$-linear combination  of classical hypergeometric series. In fact each of the classical hypergeometric series are congruent to truncated series when we reduce mod $p$. Thus, we have proved the following theorem.
\begin{thm}
\begin{align*}\#C^4_{\lambda}\equiv1-&c_{1,p}(\lambda){}_2F_1\left(\left.\begin{array}{cc}
                1/4&3/4\\
		{}&1/2
               \end{array}\right|\lambda\right)_{tr(p)}\\&-c_{2,p}(\lambda){}_2F_1\left(\left.\begin{array}{cc}
                1/2&1/2\\
		{}&1
               \end{array}\right|\lambda\right)_{tr(p)}-c_{3,p}(\lambda){}_2F_1\left(\left.\begin{array}{cc}
                3/4&5/4\\
		{}&3/2
               \end{array}\right|\lambda\right)_{tr(p)} \pmod{p},\end{align*}
where $\#C^4_{\lambda}$ is the number of $\mathbb F_p$-points plus the point at infinity and $ c_{i,p}(\lambda)\in\mathbb C$.
\end{thm}

This gives us the number of points on the curve modulo the order of the field we are working over. We have not solved for the coefficients $c_{i,p}(\lambda)$, though it is perhaps possible to using methods similar to Clemens' exposition on Legendre elliptic curves in \cite[Section 2.11]{Clemens}. Rather than go through this computation, we instead compute the exact number of points on $C^4_\lambda$ using character sums. 

\begin{thm}\label{thm:GenLegendreCurvePtCount}
 Let $q$ be a prime power such that $q\equiv 1 \pmod 4$. Let $T\in \widehat{\mathbb F_q^{\times}}$ be a generator of the character group and let $\psi=T^{\frac{q-1}{4}}$. Then
 $$\#C^4_{\lambda}=q+1+q\epsilon({\lambda})\sum_{m=1}^3\psi^m(-1)\cdot{}_2F_1\left(\left.\begin{array}{cc}
								    \psi^{-m}&\psi^{m}\\
								    {}&\psi^{2m}
								    \end{array}\right|{\lambda}\right)_{q}. $$
\end{thm}

\begin{rem*}
 This result may follow from \cite[Theorem 11]{Swisher2016}, though our equation for the generalized Legendre curve is written in a slightly different form. In \cite{Swisher2016}, the genus 3 generalized Legendre curve is written as
 $$y^4=x(1-x)(1-\lambda x).$$
 The resulting point count formulas are identical, so we should be able to find a transformation between the two curves.
\end{rem*}

\begin{proof}
To prove the result, we first express the number of points as a sum of characters over the finite field $\mathbb F_q$.
\begin{align*}
\#C^4_{\lambda}(\mathbb F_q)&=\sum_{x\in\mathbb F_q} \#\left\{y\in\mathbb F_q \left| y^4=x(x-1)(x-\lambda) \right. \right\}+1\\
			 &=\sum_{x\in\mathbb F_q-\{0,1,\lambda\}} \left(\sum_{m=0}^3 \psi^m(x(x-1)(x-\lambda))\right)+1+3\\
			 &=\sum_{x\in\mathbb F_q-\{0,1,\lambda\}} \epsilon(x(x-1)(x-\lambda))+\sum_{x\in\mathbb F_q-\{0,1,\lambda\}} \left(\sum_{m=1}^3 \psi^m(x(x-1)(x-\lambda))\right)+4\\
			 &=q-3+\sum_{x\in\mathbb F_q-\{0,1,\lambda\}} \left(\sum_{m=1}^3 \psi^m(x(x-1)(x-\lambda))\right)+4\\
\end{align*}

For each $m$ we have
\begin{align*}
 \sum_{x\in\mathbb F_q-\{0,1,\lambda\}}\psi^m(x(x-1)(x-\lambda))&=\sum_{x\in\mathbb F_q-\{0,1,\lambda\}}\psi^m(x)\psi^m(x-1)\psi^m(x-\lambda).\\
 \end{align*}
 We work to rewrite the summand and get
 \begin{align*}
								&=\sum_{x\in\mathbb F_q-\{0,1,\lambda\}}\psi^m(x)\psi^m(1-x)\psi^m(\lambda-x)\psi^m(-1)\psi^m(-1)\\
								&=\sum_{x\in\mathbb F_q-\{0,1,\lambda\}}\psi^m(x)\psi^m(1-x)\psi^m(1-\frac{1}{\lambda}x)\psi^m(\lambda)\\
								&=\psi^m(\lambda)\sum_{x\in\mathbb F_q-\{0,1,\lambda\}}\psi^m(x)\psi^m(1-x)\psi^m(1-\frac{1}{\lambda}x)\\
								&=\psi^m(\lambda)\sum_{x\in\mathbb F_q-\{0,1,\lambda\}}\psi^m(x)\psi^{-m}\psi^{2m}(1-x)\psi^m(1-\frac{1}{\lambda}x),
\end{align*}
which we recognize as being the hypergeometric function expression
\begin{align*}
								&=\psi^m(\lambda)\cdot\frac{q}{\epsilon\left(\frac{1}{\lambda}\right)\psi^{3m}(-1)}\cdot{}_2F_1\left(\left.\begin{array}{cc}
								    \psi^{-m}&\psi^{m}\\
								  {}&\psi^{2m}
								  \end{array}\right|\frac{1}{\lambda}\right)_{q}\\
								&=\psi^m(-\lambda)\cdot\frac{q}{\epsilon\left(\frac{1}{\lambda}\right)}\cdot{}_2F_1\left(\left.\begin{array}{cc}
								    \psi^{-m}&\psi^{m}\\
								    {}&\psi^{2m}
								    \end{array}\right|\frac{1}{\lambda}\right)_{q}.
\end{align*}
We use Theorem \ref{thm:Greenetransform} to write
\begin{align*}
{}_2F_1\left(\left.\begin{array}{cc}
								    \psi^{-m}&\psi^{m}\\
								    {}&\psi^{2m}
								    \end{array}\right|\frac{1}{\lambda}\right)_{q} &= \psi^{-m+m+2m}(-1)\psi^m\left(\frac{1}{\lambda}\right)\cdot{}_2F_1\left(\left.\begin{array}{cc}
								    \psi^{-m}&\psi^{-m+2m}\\
								    {}&\psi^{-m-m}
								    \end{array}\right|{\lambda}\right)_{q}\\
								    &= \psi^m\left(\frac{1}{\lambda}\right)\cdot{}_2F_1\left(\left.\begin{array}{cc}
								    \psi^{-m}&\psi^{m}\\
								    {}&\psi^{-2m}
								    \end{array}\right|{\lambda}\right)_{q}\\
								    &= \psi^m\left(\frac{1}{\lambda}\right)	\cdot{}_2F_1\left(\left.\begin{array}{cc}
								    \psi^{-m}&\psi^{m}\\
								    {}&\psi^{2m}
								    \end{array}\right|{\lambda}\right)_{q}.\\
\end{align*}
Thus, for each $m$ we have
\begin{align*}
 \psi^m(-\lambda)\cdot\frac{q}{\epsilon\left(\frac{1}{\lambda}\right)}\cdot{}_2F_1\left(\left.\begin{array}{cc}
								    \psi^{-m}&\psi^{m}\\
								    {}&\psi^{2m}
								    \end{array}\right|\frac{1}{\lambda}\right)_{q}&= \psi^m(-\lambda)\cdot\frac{q}{\epsilon\left(\frac{1}{\lambda}\right)}\cdot \psi^m\left(\frac{1}{\lambda}\right)\cdot{}_2F_1\left(\left.\begin{array}{cc}
								    \psi^{-m}&\psi^{m}\\
								    {}&\psi^{2m}
								    \end{array}\right|{\lambda}\right)_{q}\\
								    &= q\cdot \psi^m(-1)\epsilon({\lambda})\cdot{}_2F_1\left(\left.\begin{array}{cc}
								    \psi^{-m}&\psi^{m}\\
								    {}&\psi^{2m}
								    \end{array}\right|{\lambda}\right)_{q}.\\
\end{align*}

Putting this back into the formula for $\#C^4_{\lambda}$ gives

\begin{align*}
 \#C^4_{\lambda}(\mathbb F_q)&=q-3+\sum_{m=1}^3q\cdot \psi^m(-1)\epsilon({\lambda})\cdot{}_2F_1\left(\left.\begin{array}{cc}
								    \psi^{-m}&\psi^{m}\\
								    {}&\psi^{2m}
								    \end{array}\right|{\lambda}\right)_{q} +4\\
			&=q+1+q\epsilon({\lambda})\sum_{m=1}^3\psi^m(-1)\cdot{}_2F_1\left(\left.\begin{array}{cc}
								    \psi^{-m}&\psi^{m}\\
								    {}&\psi^{2m}
								    \end{array}\right|{\lambda}\right)_{q}. 
\end{align*}

\end{proof}

\subsection{Period - Point Count Connection}\label{sec:periodpointcount}
We notice two phenomena here that also occur when one computes periods and point counts for Legendre elliptic curves. The first is that, remarkably, the number of points on the curve can be expressed in terms of finite field hypergeometric functions with input given by $\lambda$. In fact we get equality, not just a congruence, between the number of points and a finite field hypergeometric expression.  This phenomenon also occurs for families of curves not expressible in Legendre form over $\mathbb Q$ (see, for example, \cite{FOP2, Fuselier10, Lennon1}). In fact, this phenomenon seems to extend to some higher dimensional Calabi-Yau manifolds as is shown in \cite{AhlgrenOno00a, GoodsonDwork2017, GoodsonDwork2017b, McCarthy2012a, RV1} leading us to wonder if this will be the case for a large class of algebraic varieties.\\

The second phenomenon is that in computing the point count in two different ways, we get a congruence between classical and finite field hypergeometric expressions. We can say a bit more on this: it seems as though we can identify congruences between particular summands that ``match''.  For example, we saw in Theorem \ref{thm:GenLegendreCurvePeriod} that one period of the curve $C^4_\lambda$ can be expressed as 
$$\pi_2={}_2F_1\left(\left.\begin{array}{cc}
                1/2&1/2\\
		{}&1
               \end{array}\right|\lambda\right).$$               
We saw in Theorem  \ref{thm:GenLegendreCurvePtCount} that one of the summands in the point count for $C^4_{\lambda}$ is
$$qT^{\frac{q-1}{2}}(-1){}_2F_1\left(\left.\begin{array}{cc}
                T^{\frac{q-1}{2}}&T^{\frac{q-1}{2}}\\
		{}&\epsilon
               \end{array}\right|\lambda\right)_{q}.$$               

Note that since $q\equiv 1 \pmod 4$, $T^{\frac{q-1}{2}}(-1)=1$. Theorem \ref{theorem:2F1congruence} then tells us that this expression is congruent modulo $p$ to the negative of the classical hypergeometric series $\pi_2$.  The classical and finite field hypergeometric expressions -- including the two that are not covered by Theorem \ref{theorem:2F1congruence} -- for the generalized Legendre curves ``match'' in the same way that period and trace of Frobenius expressions match for elliptic curves: we replace the fraction $\frac{a}{b}$ with a character of order $b$ raised to the $a$th power. This phenomenon also seems to extend to some other curves (see \cite{Mortenson2003a}) and to higher dimensional Calabi-Yau manifolds (see, for example, \cite{GoodsonDwork2017, GoodsonDwork2017b, Kilbourn2006, McCarthy2012a, Mortenson2005}. By testing values in Sage \cite{Sage}, we know that it is not the case that congruences exist between arbitrary (matching) truncated hypergeometric series and finite field hypergeometric functions. This leads us to wonder when we can expect to have a congruence between these two types of series.

\bibliographystyle{plain}
\bibliography{LegendreBib}

\end{document}